\documentclass[12pt]{amsart}
\usepackage[all]{xy}
\usepackage{amssymb}
\usepackage{amsthm}
\usepackage{amsmath}
\usepackage{amscd,enumitem}
\usepackage{verbatim}
\usepackage{eurosym}
\usepackage{graphicx}
\usepackage{dsfont}
\usepackage{stmaryrd}
\usepackage{color}
\usepackage{float}
\usepackage{longtable}
\usepackage{dcolumn}
\usepackage[mathscr]{eucal}
\usepackage[all]{xy}
\usepackage{hyperref}
\usepackage[title]{appendix}
\usepackage[usenames,dvipsnames]{xcolor}
\usepackage{bbm}
\usepackage[textheight=8.85in, textwidth=6.8in]{geometry}
\usepackage{multirow}
\usepackage{caption}
\theoremstyle{plain}
\newtheorem{theorem}{Theorem}[section]

\newtheorem{lemma}[theorem]{Lemma}

\theoremstyle{definition}
\newtheorem{definition}[theorem]{Definition}

\theoremstyle{remark}

\newcommand{\Z}{\mathbb{Z}}
\newcommand{\Q}{\mathbb{Q}}

\newcommand{\F}{\mathbb{F}}

\newcommand{\hes}{\text{{\rm{Hes}}}}
\newcommand{\jac}{\text{{\rm{Jac}}}}
\newcommand{\dik}{\text{{\rm{DIK}}}}

\newcommand{\R}{\mathbb{R}}

\newcommand{\C}{\mathbb{C}}

\catcode`,\active

\catcode`\,12

\numberwithin{equation}{section}

\begin{document}
	\title[$p$-adic hypergeometric functions and Elliptic curves]{Weighted averages of $p$-adic hypergeometric functions and traces of Frobenius of elliptic curves}
		\author{Riya Mandal and Neelam Saikia}

	\address{School of Basic Sciences, Indian Institute of Technology Bhubaneswar, Argul, Khordha 752050, Odisha, India.}
	\email{a23ma09010@iitbbs.ac.in}

	\address{School of Basic Sciences, Indian Institute of Technology Bhubaneswar, Argul, Khordha 752050, Odisha, India.}
	\email{neelamsaikia@iitbbs.ac.in}

	
	\begin{abstract}
		In this paper, we aim to study traces of Frobenius of certain one parameter families of elliptic curves and their relationships with $p$-adic hypergeometric functions. For example, we consider a DIK family of curves and establish the trace of Frobenius as weighted averages of special values of certain families of $p$-adic hypegeometric functions, where the average is taken over the arrays of parameters. Moreover, we consider Jacobi curves and express the trace of Frobenius as a special values of $p$-adic hypergeomtric functions. As a consequence of these results we obtain four summation identities for the $p$-adic hypegeometric functions that arise from the DIK family. Furthermore, we obtain $p$-adic analogous of Euler and Pfaff transformations for certain $p$-adic hypergemetric functions.
	
         \end{abstract}
	\maketitle
	\section{Introduction and Statement of results}
	Let $p$ be an odd prime and $\F_p$ be a finite field with $p$ elements. In 1980's, Greene \cite{greene} introduced a class of functions defined over finite fields known as Gaussian hypergeometric functions which can be described as finite field analogues of classical hypergeometric series.  
	He defined these functions using Jacobi sums of multiplicative characters over finite fields $\F_p.$ Furthermore, he established that these functions satisfy many transformation and special identities as satisfied by the classical hypergeometric series. Similar type of questions were also studied by other authors for Greene's hypergeometric functions. For more details, see \cite{Evan-Greene1, Evan-Greene2, 
	Greene-Stanton} etc. These hypergeometric functions often arise in the study of arithmetic of elliptic curves, Galois representations, fourier coefficients of modular forms, Kloosterman sums etc. For example, Ono \cite{ono} derived that the trace of Frobenius of certain families of elliptic curves over finite fields can be expressed as special values of Gaussian hypergeometric functions. In this direction, further results can be found in \cite{ahlgren-ono, AO, B-K, fop, Fuselier, FuselierHecke, Lennon1, Lennon, Long, MP, Rouse, TY}. By definition of Gaussian hypergeometric function the parameters are multiplicative characters of finite field. To facilitate the existence of multiplicative characters of order higher than two over a finite field $\F_p,$ it is necessary that the prime $p$ has to satisfy certain congruence. Therefore, the results involving Gaussian hypergeometric functions with higher order characters as parameters are certainly limited to certain congruence conditions.  For example, see \cite{Fuselier, FuselierHecke, Lennon1, Lennon}. To overcome this restriction McCarthy \cite{McCarthyPJM} introduced a new version of hypergeometric functions over finite field using $p$-adic gamma functions extending Gaussian hypergeometric functions in the $p$-adic setting. McCarthy, also derived that special values of this hypergeometric function in the $p$-adic setting correspond to traces of Frobenius of elliptic curves for all but finitely many primes $p.$ In \cite{BS, BS1, BSM, Barman-Sul2}, the authors developed arithmetic of further families of hypergeometric functions in the $p$-adic setting. Further results in this direction relating fourier coefficients of modular forms can be found in \cite{FM, Pujahari-Saikia, Barman-Sul}. Recently, Pujahari and the second author \cite{Pujahari-Saikia2}, investigated statistical questions determining limiting distributions of certain families of hypergeometric functions in the $p$-adic setting as $p\rightarrow\infty$ confirming semi circular or Sato-Tate distribution for these families. This result is a refinement of the result of Birch \cite{Birch} that confirm Sato-Tate distribution for families of elliptic curves over finite fields in the vertical setting. We now fix some notation to recall hypergeometric function in the $p$-adic setting. Let $\Q_p$ be the field of $p$-adic numbers and $\Z_p$ be the ring of $p$-adic integers. For $x\in\R,$ $\lfloor x \rfloor$ denotes the greatest integer function and $\langle x\rangle$ denotes the fractional part of $x$ satisfying $0\leq x<1.$ Let $\omega$ denote the Teichm\"{u}ller character of $\F_p$ such that $\omega (a)\equiv a\pmod{p}$ and 
	$\overline{\omega}$ denote the multiplicative inverse of $\omega.$ Furthermore, let $\Gamma_p(\cdot)$ denote the Morita's $p$-adic gamma function (see Section 2 for more details). With these notation, we now recall the following hypergeometric function which is known as $p$-adic hypergeometric function:
	\begin{definition}\cite[Definition 1.1]{McCarthyPJM} 
	 Let $p$ be an odd prime and let $t\in \mathbb{F}_p$. For $n\in\mathbb{Z}^+$ and $1\leq i\leq n,$ let 
    $a_i,b_i \in \mathbb{Q}\cap \mathbb{Z}_p $. Then define
    \begin{align*}
        {_nG_n}\left[\begin{array}{cc}
        a_1,    a_2, \dots,a_n   \\ b_1,   b_2, \dots,b_n
       \end{array}|t\right]_p :&=\frac{-1}{(p-1)} \sum_{j=0}^{p-2}(-1)^{jn}\overline{\omega}^{j}(t)\\& \prod_{i=1}^{n}(-p)^{-\lfloor\langle a_i\rangle-\frac{j}{p-1}\rfloor-\lfloor\langle-b_i\rangle+\frac{j}{p-1}\rfloor}
\frac{\Gamma_p(\langle a_i-\frac{j}{p-1}\rangle)}{\Gamma_p(\langle a_i\rangle)}\frac{\Gamma_p(\langle-b_i+\frac{j}{p-1}\rangle )}{\Gamma_p(\langle-b_i\rangle)} .
        \end{align*}
	\end{definition}
	In this paper, our goal is to explore one parameter families of elliptic curves and investigate the associated families of hypergeometric functions. For example, we consider the following family of elliptic curve: For $\lambda\in\Q$ such that $\lambda\neq0,\frac{9}{4},$ let 
	$$
	 E^{\dik}_{\lambda} : y^2 = x^3 + 3\lambda(x + 1)^2 
	  $$ 
	  be the DIK elliptic curve with discriminant, $\Delta(E_{\lambda}^{\dik})=432\lambda^2(4\lambda-9)$ and $j$-invariant, $j(E_{\lambda}^{\dik})=\frac{4\lambda\cdot1728(\lambda-2)^3}{4\lambda-9}.$
	For each prime $p$ of good reduction, let $E_{\lambda}^{\dik}(\F_p)$ denote the set of $\F_p$-rational points of the curve $E_{\lambda}^{\dik}$ after reduction modulo $p$ including the point at infinity and $a_p(E_{\lambda}^{\dik}):=p+1-|E_{\lambda}^{\dik}(\F_p)|$ be its trace of Frobenius, where $|E_{\lambda}^{\dik}(\F_p)|$ denote the number of elements present in the set $E_{\lambda}^{\dik}(\F_p).$ Furthermore, we consider Jacobi and Hessian curves defined as follows, respectively: For $\lambda\neq\pm1,0,$ let $$E_{\lambda}^{\jac}:y^2=x^4 + 2\lambda x^2 + 1$$ be the Jacobi curve with discriminant $\Delta(E^{\jac}_{\lambda})=256(\lambda^2-1)^2$ and for $d^3\neq1,$ let
	$$
	E^{\hes}_d: x^3+y^3+1=3dxy
	$$ 
	be the Hessian elliptic curve with $\Delta(E^{\hes}_d)=27(1-d^3)^3.$ Furthermore, let $a_p(E^{\jac}_\lambda):=p+1-|E^{\jac}_\lambda(\F_p)|$ be the trace of Frobenius of the curve $E^{\jac}_\lambda$ and $a_p(E^{\hes}_d):=p+1-|E^{\hes}_d(\F_p)|$ be the trace of Frobenius of the curve $E^{\hes}_d,$ where $|E^{\jac}_\lambda(\F_p)|$ and $|E^{\hes}_d(\F_p)|$ denote the number of $\F_p$-points of the respective curves after reduction modulo $p$ including the points at infinity. To this end we now define certain families of $p$-adic hypergeometric functions: For $0\leq r\leq p-2$ and $t\in\F_p,$ let 
	$${_2G_2}(r,t)_p:={_2G_2}\left[\begin{array}{cc}
  \frac{r}{p-1}   & -\frac{1}{2} -\frac{r}{p-1}  \\
  \frac{-r}{2(p-1)}    & -\frac{r}{2(p-1)} -\frac{1}{2} \end{array}| t \right]_p$$
  and for $0\leq s\leq p-2,$ let $$\tilde{_2G_2}(s,t)_p:={_2G_2}\left[\begin{array}{cc}
  0   & -\frac{1}{2} -\frac{s}{p-1}  \\
  \frac{s}{(p-1)}    & -\frac{2s}{(p-1)} \end{array}|\ \frac{1}{6t} \right]_p.$$
	Using these notation, we state our first result expressing the weighted average of the $p$-adic hypergeometric function ${_2G_2}(r,1)_p$ as $r$ varying in terms of the trace of Frobenius of the DIK curve. 
	
	\begin{theorem}\label{theorem-DIK-1}
For a prime $p>3$ and $\lambda\neq0,\frac{9}{4}$ we have 
	\begin{align}
\frac{p^2\phi(-3\lambda)}{(p-1)} \sum_{r=1}^{p-2}{\omega}^r(6\lambda)  \binom{\overline{\omega}^r}{\overline{\omega}^r\phi}
{_2G_2}(r,1)_p=\frac{\phi(3\lambda)}{(p-1)} + a_p(E_{\lambda}^{\dik}).\notag
   \end{align}
	\end{theorem}
	In the next result, we express the weighted average of the $p$-adic hypergeometric function $\tilde{_2G_2}(s,t)_p$ as $s$ varying in terms of the trace of Frobenius of the DIK curve. 
	\begin{theorem}\label{theorem-DIK-2}
	For a prime $p>3$ and $\lambda\neq0,\frac{9}{4}$ we have 
	\begin{align}
	p\sum\limits_{s=1,\ s\neq\frac{p-1}{2}}^{p-2}\tilde{_2G_2}(s,\lambda)_p+(p+1)+p \phi(1-6\lambda)(1+\phi(-6\lambda))=(1-p)\cdot\phi(3\lambda)\cdot a_p(E_{\lambda}^{\dik}).\notag
  \end{align}
 	\end{theorem}
	
	In the next two theorems, we provide two summation formulas for the functions ${_2G_2}(r,t)_p$ and $\tilde{_2G_2}(s,t)_p$ as a special value of a $p$-adic hypergeometric functions. 
	
	\begin{theorem}\label{theorem-summation-1}
	Let  $p>3$ be a prime and $\lambda\in \mathbb{F}_p$ such that $\lambda\neq0,2,\frac{9}{4},\frac{3\pm\sqrt{3}}{2}.$ Then we have 
	\begin{align} 
	&\frac{p^2}{(p-1)} \sum_{r=1}^{p-2}{\omega}^r(6\lambda)  \binom{\overline{\omega}^r}{\overline{\omega}^r\phi}{_2G_2}(r,1)_p\notag\\&= \frac{\phi(-1)}{p-1}
  +p\cdot\phi((-3)(2\lambda^2-6\lambda+3)) \cdot{_2G_2}\left[\begin{array}{cc}
  \frac{1}{4}   & \frac{3}{4}  \\
  \frac{1}{3}    & \frac{2}{3}  \end{array}|\frac{(2\lambda^2-6\lambda+3)^2}{4\lambda(\lambda-2)^3}\right]_p.\notag
   \end{align}
	\end{theorem}

	\begin{theorem}\label{theorem-summation-2}
	Let  $p>3$ be a prime and $\lambda\in \mathbb{F}_p$ such that $\lambda\neq0,2,\frac{9}{4},\frac{3\pm\sqrt{3}}{2}.$ Then we have 
\begin{align*}
	&\sum\limits_{s=1,\ s\neq\frac{p-1}{2}}^{p-2}\tilde{_2G_2}(s,\lambda)_p+\frac{(p+1)}{p}+ \phi(1-6\lambda)(1+\phi(-6\lambda))\\&=\frac{(1-p)}{\phi(6\lambda^2-18\lambda+9)} \cdot {}_2G_2\left[\begin{array}{cc}
  \frac{1}{4}   & \frac{3}{4}  \\
  \frac{1}{3}    & \frac{2}{3}  \end{array}|\frac{(2\lambda^2-6\lambda+3)^2}{4\lambda(\lambda-2)^3}\right]_p.
\end{align*}
	\end{theorem}
	In the next theorem, we express the trace of Frobenius of the Jacobi curve as special values of $p$-adic hypergeometric functions.
	
	\begin{theorem}\label{theorem-Jacobi}
	Let $p>3$ be a prime and $\lambda\in \mathbb{F}_p$ such that $\lambda\neq0, \pm 1.$ Then we have 
	$$ 
	a_p(E^{\jac}_{\lambda})=1+\phi(2\lambda) \cdot{_2G_2}\left[\begin{array}{cc}
\frac{1}{4}   & \frac{3}{4} \\
  0    & 0  \end{array}|{\lambda^2}\right]_p=1+p\cdot\phi(-2\lambda)\cdot {_2G_2}\left[\begin{array}{cc}
\frac{1}{4}   & \frac{3}{4} \\
  \frac{1}{2}    & \frac{1}{2}  \end{array}|{\lambda^2}\right]_p.$$
\end{theorem}
Another, important question that we aim to investigate is to find transformation formulas of $p$-adic hypergeometric functions that are analogous to the classical hypergeometric identities. For example, the second equality that arise in Theorem \ref{theorem-Jacobi} can be described as a $p$-adic analogue of Euler's linear transformation identity. For a detailed study of such transformation identities we refer \cite{Askey, bailey}. Moreover, we provide an identity that is analogous to Pfaff transformation in the next theorem.

	\begin{theorem}\label{theorem-Jacobi trans}
	Let $p>3$ be a prime and $\lambda\in \mathbb{F}_p$ such that $\lambda\neq0, \pm 1, \pm\sqrt{2}.$ Then we have 
	$$ 
	\phi(-2\lambda){_2G_2}\left[\begin{array}{cc}
\frac{1}{4}   & \frac{3}{4} \\
  \frac{1}{2}    & \frac{1}{2}  \end{array}|{\lambda^2}\right]_p=\phi(1-\lambda^2){_2G_2}\left[\begin{array}{cc}
\frac{1}{4}   & \frac{3}{4} \\
  \frac{1}{2}    & \frac{1}{2}  \end{array}|\frac{\lambda^2}{\lambda^2-1}\right]_p.$$ 
  \end{theorem}
	
In the following two theorems, we state further summation identities for the functions ${_2G_2}(r,t)_p$ and $\tilde{_2G_2}(s,t)_p.$	
	
\begin{theorem}\label{Hessian-trans}
	Let $p>3$ be a prime and  $d\in \mathbb{F}_p$ be such that $d\neq-2,0,$  $d^3\neq 1$ and $\frac{d^2+d+1}{3(d+2)}$ is square in $\mathbb{F}_p^\times$, then  
   \begin{align*}
&\frac{p^2\phi(-(3d^3+9d^2+9d+6))}{(p-1)} \sum_{r=1}^{p-2}{\omega}^r\left(\frac{3(d+2)^3}{2(d^2+d+1)}\right)  \binom{\overline{\omega}^r}{\overline{\omega}^r\phi}{_2G_2}(r,1)_p\\&=1-\gamma-N_0+\frac{\phi(3d^3+9d^2+9d+6)}{(p-1)} +p\cdot \phi(-3d)\cdot{_2G_2}\left[\begin{array}{cc}
  \frac{1}{2}   & \frac{1}{2}  \\
  \frac{1}{6}    & \frac{5}{6}  \end{array}|\frac{1}{d^3}\right]_p,
\end{align*}
where $$\gamma=\begin{cases}
    5-6\phi(-3), & \text{if} \ p \equiv 1 \pmod{3},\\
    1, & \text{if} \ p\not\equiv 1 \pmod{3},
\end{cases}\
\text{and} \ N_0=\begin{cases}
    2, & \text{if} \ p \equiv 1,-5 \pmod{12},\\0, &  \text{if}\ p\not\equiv 1,-5 \pmod{12}.
\end{cases}$$
\end{theorem}

\begin{theorem}\label{Hessian-trans2}
	Let $p>3$ be a prime and  $d\in \mathbb{F}_p$ be such that $d\neq-2,0,$ $d^3\neq 1$ and define $t:=\frac{d^2+d+1}{3(d+2)}$ such that $t$ be square in $\mathbb{F}_p^\times$, then  
  \begin{align*}
&\frac{p\cdot\phi(3d^3+9d^2+9d+6)}{(p-1)}\sum\limits_{s=1,\ s\neq\frac{p-1}{2}}^{p-2}\tilde{_2G_2}\left(s,\frac{(d+2)^2}{12t}\right)_p+\frac{(p+1)\phi(3d^3+9d^2+9d+6)}{p-1}\\&+\frac{p\cdot\phi(3d^3+9d^2+9d+6)\phi(-3d^3-16d^2-34d-22)(1+\phi(-(6d^3+18d^2+18d+12)))}{(p-1)\phi(2(d^2+d+1))}\\&=-1+\gamma +N_0 -p\cdot\phi(-3d)\cdot{_2G_2}\left[\begin{array}{cc}
  \frac{1}{2}   & \frac{1}{2}  \\
  \frac{1}{6}    & \frac{5}{6}  \end{array}|\frac{1}{d^3}\right]_p,
\end{align*}
where $\gamma$ and $N_0$ are same as defined in Theorem \ref{Hessian-trans}.
\end{theorem}

The rest of the paper is structured as follows: In Section 2, we recall some basic definitions, notation and some important theorems including Hasse-Davenport relation and Gross-Koblitz theorem. We also, derive some preliminary lemmas in Section 2 that are useful in the proof of main results. In Section 3, we prove the main theorems.

	
	\section{Notation and preliminary results}
	In this section we recall some basic definitions and some useful lemmas and some important theorems.  We begin with some results of multiplicative characters.

	\subsection{Multiplicative characters}

 Let $\widehat{\F_p^{\times}}$ be the set of all multiplicative characters of $\F_p^{\times}$.  We extend the domain of definition of each character to $\mathbb{F}_p$ by simply setting $\chi(0)=0$  including the trivial character $\varepsilon$. The following orthogonality relation of multiplicative characters is very useful in our calculations:
 
 \begin{lemma}\cite[Chapter 8]{rosen}\label{orthogonality}
 For a multiplicative character $\chi$ of $\F_p,$ we have the following
 \begin{equation}
    \sum_{x\in\mathbb{F}_p}\chi{(x)}=\begin{cases} p-1 & \text{if}\   \chi=\varepsilon;\\ 
    0  & \text  {if} \  \chi \neq \varepsilon,\\\
\end{cases}
\end{equation}
and 
\begin{equation}
\sum_{\chi\in\widehat{{\mathbb{F}}_p^\times}}\chi{(x)}=
\begin{cases}p-1 & \text  {if}\   x=1;\\ 
0  & \text  {if} \  x \neq 1.
\end{cases}
\end{equation}
 \end{lemma}
 Let $\zeta_p$ be the primitive $p$-th roots of unity. The additive character $\theta:\mathbb F_p\rightarrow\C^\times$ is defined by $\theta(x)=\zeta_p^x.$ 
 For a multiplicative character $\chi,$ Gauss sum, $g(\chi)$ is defined by $$g(\chi):=\sum_{x\in \mathbb F_p}\chi(x)\theta(x).$$  It is easy to see that $g(\varepsilon)= -1.$ Furthermore, Gauss sums satisfy the following identity. If $\chi\neq\varepsilon,$ then 
\begin{equation}\label{inverse}
    g(\chi)g(\overline{\chi})=
        p\cdot\chi(-1). 
        \end{equation}
        \begin{lemma}\label{lemma-fuselier} \cite[Lemma 2.2]{Fuselier}
        For $\alpha\in\F_p^{\times}$ we have 
        $$
        \theta(\alpha)=\frac{1}{p-1}\sum\limits_{\chi\in\widehat{\F_p^{\times}}}g(\overline{\chi})\chi(\alpha).
        $$
        \end{lemma}
Further details on Gauss sums can be found in \cite{Gauss-Jacobi, rosen}. The following product formula of Gauss sums is due to Hasse and Davenport. 

\begin{theorem}\cite[Hasse-Davenport relation, Theorem 11.3.5]{Gauss-Jacobi} \label{Hasse} 
Let $\chi$ be a multiplicative character of order $m$ of $\mathbb{F}_p$ for some positive integer $m$ dividing $p-1$. For any multiplicative character $\psi$ of $\mathbb{F}_p$ we have 
$$\prod\limits_{i=0}^{m-1}g(\chi^i\psi)=g(\psi^m)\psi^{-m}(m)\prod\limits_{i=1}^{m-1}g(\chi^i).$$
    
\end{theorem}

Another important character sum is Jacobi sum. We now recall the definition of Jacobi sum and its relationship with Gauss sum. For two multiplicative characters $A,B$ of $\F_p,$ the Jacobi sum, $J(A,B)$ is defined as 
$$
J(A,B):=\sum\limits_{x\in\F_p}A(x) B(1-x).
$$
We have the following giving a nice relation between Gauss and Jacobi sums. 
\begin{lemma}\cite{greene}\label{gauss jacobi relation} Let $A,B$ be two multiplicative characters of $\mathbb{F}_p$ such that $AB\neq \varepsilon$. Then we have 
\begin{align}
J(A, B)=\frac{g(A)g(B)}{g(AB)}.\notag
\end{align}
\end{lemma} 

For two multiplicative characters $A,B$ of $\F_p,$ the binomial coefficient ${A\choose B}$ is defined as $${A\choose B}=\frac{B(-1)}{p}J(A, \overline{B}).$$


\subsection{p-adic preliminaries}

Let $\overline{\mathbb{Q}}_p$ denote the algebraic closure of $\mathbb{Q}_p$ and $\mathbb{C}_p$ denote its the completion. It is well known that $\Z_p$ contains all the $(p-1)$-th roots of unity. Therefore, we may consider the multiplicative character $\chi:\mathbb{F}_p\rightarrow \mathbb{Z}_p.$  For a positive integer $n,$ the $p$-adic gamma function $\Gamma_p(n)$ is defined as
\begin{align}
\Gamma_p(n):=(-1)^n\prod\limits_{0<j<n,p\nmid j}j.\notag
\end{align}
The domain of definition of $p$-adic gamma function can be extended to all $x\in\mathbb{Z}_p$ by setting $\Gamma_p(0):=1$ and for $x\neq0$
\begin{align}
\Gamma_p(x):=\lim_{x_n\rightarrow x}\Gamma_p(x_n),\notag
\end{align}
where $(x_n)$ is a sequence of positive integers $p$-adically approaching to $x.$
For more details on $p$-adic analysis, see \cite{kob}. We now recall some basic properties of $p$-adic gamma functions from \cite{kob}.
For $x\in \mathbb{Z}_p$, we have 
\begin{align}\label{prod-1}
    \Gamma_p(x)\Gamma_p(1-x)=(-1)^{x_0},
\end{align}
where $x_0 \equiv x \pmod{p}$ and $ x_0\in\{1,2,3,\ldots ,p\}.$ Let $m\in \mathbb{Z}^+$ such that $p\nmid m$ and let $x=\frac{r}{p-1}$ with $0\leq r \leq p-1$ then 
\begin{align}\label{prod-2}
   \prod\limits_{h=0}^{m-1} \Gamma_p\left(\frac{x+h}{m}\right)=\omega(m^{(1-x)(1-p)})\Gamma_p(x)\prod\limits_{h=1}^{m-1}\Gamma_p\left(\frac{h}{m}\right).
\end{align}	
A refinement of \eqref{prod-2} is the following lemma.
\begin{lemma}\cite[Lemma 4.1]{McCarthyPJM} Let $p$ be an odd prime and $t\in \mathbb{Z}^+$ such that $p \nmid t$. For $0\leq j \leq p-2$, we have 
\begin{align}\label{prod-3}
   \Gamma_p\left(\left\langle\frac{tj}{p-1}\right\rangle\right)\omega(t^{tj}) \prod\limits_{h=1}^{t-1}\Gamma_p\left(\frac{h}{t}\right)=\prod\limits_{h=0}^{t-1}\Gamma_p\left(\left\langle\frac{h}{t}+\frac{j}{p-1}\right\rangle\right).
\end{align}
\end{lemma}

The next theorem due to Gross and Koblitz \cite{gross} is very important in our purpose. It provides a relation between Gauss sum and $p$-adic gamma function. To state this result we need to fix some notation. 
Let $\pi\in\mathbb{C}_p$ be a fixed root of $x^{p-1}+p=0$ satisfying $\pi \equiv \zeta_p-1\pmod{(\zeta_p-1)^2}$. 

\begin{theorem} \cite[Gross-Koblitz formula]{gross})\label{gross-koblitz}
    For $j\in \mathbb{Z}$ we have $$g(\overline{\omega}^j)=-\pi^{(p-1)\langle\frac{j}{p-1}\rangle}\Gamma_p\left(\left\langle\frac{j}{p-1}\right\rangle\right).$$
\end{theorem}

Next we state some preliminary lemmas for the ease of calculation in the proofs of main results. 

\begin{lemma}\label{exponent-1}
For $0\leq r,s\leq p-2 $ we have 
$$\left\lfloor\frac{2s+r}{p-1}\right\rfloor=\left\lfloor\frac{r}{2(p-1)}+\frac{s}{p-1}\right\rfloor+\left\lfloor\frac{r}{2(p-1)}+\frac{s}{p-1}+\frac{1}{2}\right\rfloor.$$ 
 \end{lemma}
\begin{proof}
Since $0\leq\frac{s}{p-1}, \frac{r}{p-1}<1,$ so we have  $0\leq\frac{2s}{p-1}+\frac{r}{p-1}<3.$ Therefore, we complete the proof by considering term $\frac{2s}{p-1}+\frac{r}{p-1}$ in three disjoint intervals. First suppose that 
  $0\leq\frac{2s}{p-1}+\frac{r}{p-1}<1. $ This implies that $0\leq\frac{s}{p-1}+\frac{r}{2(p-1)}<\frac{1}{2}$ and $\frac{1}{2}\leq\frac{s}{p-1}+\frac{r}{2(p-1)}+\frac{1}{2}<1.$ Now using these three inequalities we conclude that  
 $$ \left\lfloor\frac{2s}{p-1}+\frac{r}{p-1}\right\rfloor=0\ \text{and}\ \left\lfloor\frac{r}{2(p-1)}+\frac{s}{p-1}\right]+\left\lfloor\frac{r}{2(p-1)}+\frac{s}{p-1}+\frac{1}{2}\right\rfloor=0.$$ 
  Using similar arguments we conclude the lemma for the cases $1\leq\frac{2s}{p-1}+ \frac{r}{p-1}<2$ and $2\leq\frac{2s}{p-1}+ \frac{r}{p-1}<3.$
\end{proof}
\begin{lemma}\label{exponent-2}
For $0\leq r,s \leq p-2$ we have 
 $$\left\lfloor-\frac{1}{2}-\frac{s}{p-1}-\frac{r}{p-1}\right\rfloor=\begin{cases}
    -1+\left\lfloor\left\langle-\frac{1}{2}-\frac{s}{p-1}\right\rangle-\frac{r}{p-1}\right\rfloor &  \ \text{if}\  0\leq s\leq \frac{p-1}{2}; \vspace{.1 cm}\\  
     -2+\left\lfloor\left\langle-\frac{1}{2}-\frac{s}{p-1}\right\rangle-\frac{r}{p-1}\right\rfloor & \ \text{if} \ \frac{p-1}{2}<s<p-1.
\end{cases}$$
\end{lemma}

\begin{proof}
Proceeding similar steps as in the proof of Lemma \ref{exponent-1}, we complete the proof.
\end{proof}
Furthermore, we need the following lemma.

\begin{lemma}\label{exponent-3}
For $1\leq r \leq p-2$ and $0\leq s \leq p-2$ we have $$\left\lfloor-\frac{1}{2}-\frac{r}{p-1}\right\rfloor-\left\lfloor\frac{-r-s}{p-1}-\frac{1}{2}\right\rfloor=-\left\lfloor\left\langle-\frac{1}{2}-\frac{r}{p-1}\right\rangle-\frac{s}{p-1}\right\rfloor.$$
\end{lemma}

\begin{proof}
Proceeding similar steps as in the proof of Lemma \ref{exponent-1}, we complete the proof.
\end{proof}
		
We also use the following three lemmas for the ease of calculation:

\begin{lemma}
For $0\leq r\leq p-2$ we have 
\begin{align}\label{relation-1}
    &\Gamma_p\left(\left\langle\frac{r}{2(p-1)}\right\rangle\right)\Gamma_p\left(\left\langle\frac{r}{2(p-1)}+\frac{1}{2}\right\rangle\right) \Gamma_p\left(\left\langle\frac{-r}{p-1} -\frac{1}{2}\right\rangle\right)\notag\\
    &=p\phi(-1)g(\phi)\Gamma_p\left(\frac{1}{2}\right)(-p)^{\frac{1}{2}+\lfloor-\frac{1}{2}-\frac{r}{p-1}\rfloor}{\omega}^r(-2)\binom{\overline{\omega}^r}{\overline{\omega}^r\phi}.
\end{align}
\end{lemma}	

\begin{proof}	
If we apply \eqref{prod-2} for $x=\frac{r}{p-1}$ and $m=2,$ then $$\Gamma_p\left(\left<\frac{r}{2(p-1)}\right>\right)\Gamma_p\left(\left<\frac{r}{2(p-1)}+\frac{1}{2}\right>\right)=w(2^{1-p+r})\Gamma_p\left(\left<\frac{r}{p-1}\right>\right)\Gamma_p\left(\frac{1}{2}\right).$$ Now, replacing the first two terms of the left side of \eqref{relation-1} by the above identity and then using Gross-Koblitz formula (Theorem \ref{gross-koblitz}) we simplify further to obtain the desired result.
\end{proof}

\begin{lemma}\label{relation-2}

For a prime $p>3$ we have 
\begin{align}
&\sum_{s=0}^{p-2}(-p)^{-\left\lfloor\frac{-s}{p-1}\right\rfloor-\left\lfloor\frac{s}{p-1}\right\rfloor-\left\lfloor\frac{s}{p-1}+\frac{1}{2}\right\rfloor-\left\lfloor-\frac{1}{2}-\frac{s}{p-1}\right\rfloor}\frac{\Gamma_p\left(\left\langle\frac{-s}{p-1}\right\rangle\right)\Gamma_p\left(\left\langle\frac{s}{p-1}\right\rangle\right)\Gamma_p\left(\left\langle\frac{s}{p-1}+\frac{1}{2}\right\rangle\right)\Gamma_p\left(\left\langle\frac{-s}{p-1}-\frac{1}{2}\right\rangle\right)}{\Gamma_p\left(\frac{1}{2}\right) \Gamma_p\left(\frac{1}{2}\right)}\notag\\
&=-p(p-1)(p-2).\notag
\end{align}
\end{lemma}

\begin{proof}
Let $B=\sum\limits_{s=0}^{p-2}(-p)^{-\left\lfloor\frac{-s}{p-1}\right\rfloor-\left\lfloor\frac{s}{p-1}\right\rfloor-\left\lfloor\frac{s}{p-1}+\frac{1}{2}\right\rfloor-\left\lfloor-\frac{1}{2}-\frac{s}{p-1}\right\rfloor}\frac{\Gamma_p\left(\left\langle\frac{-s}{p-1}\right\rangle\right)\Gamma_p\left(\left\langle\frac{s}{p-1}\right\rangle\right)\Gamma_p\left(\left\langle\frac{s}{p-1}+\frac{1}{2}\right\rangle\right)\Gamma_p\left(\left\langle\frac{-s}{p-1}-\frac{1}{2}\right\rangle\right)}{\Gamma_p\left(\frac{1}{2}\right) \Gamma_p\left(\frac{1}{2}\right)}.$

It is clear that the term under the summation for $s=0$ is $-p.$ If $s\neq0,$ then using Gross-Koblitz formula (Theorem \ref{gross-koblitz}) we have 
$$\Gamma_p\left(\left\langle\frac{s}{p-1}+\frac{1}{2}\right\rangle\right)\Gamma_p\left(\left\langle\frac{-s}{p-1}-\frac{1}{2}\right\rangle\right)=(-p)^{\left\lfloor\frac{s}{p-1}+\frac{1}{2}\right\rfloor+\left\lfloor-\frac{1}{2}-\frac{s}{p-1}\right\rfloor}
g(\overline{\omega}^s\phi)g({\omega}^{s}\phi).$$
Now, substituting the above identity in the expression of $B$ and using \eqref{prod-1} we simplify further to obtain 
\begin{align}
B&=-p+\sum_{s=1,\ s\neq \frac{p-1}{2}}^{p-2} \frac{p\ \overline{\omega}^s(-1)p\ {\overline{\omega}^s\phi}(-1)}{\Gamma_p(\frac{1}{2})\Gamma_p(\frac{1}{2})}+\frac{p\phi(-1)}{\Gamma_p(\frac{1}{2})\Gamma_p(\frac{1}{2})}
=-p(p-1)(p-2).\notag
\end{align}
\end{proof}
\begin{lemma}\label{lemma-100}
For $1\leq s\leq p-2$ such that $s\neq \frac{p-1}{2},$ we have 
$$\Gamma_p\left(\left\langle\frac{-s}{p-1}-\frac{1}{2}\right\rangle\right)\Gamma_p\left(\left\langle\frac{-s}{p-1}\right\rangle\right) \Gamma_p\left(\left\langle\frac{2s}{p-1}\right\rangle\right)=-{\Gamma_p\left(\frac{1}{2}\right)}\overline{\omega}^s(4).$$
\end{lemma}
\begin{proof}
Using \eqref{prod-2} for $x=\left\langle\frac{-2s}{p-1}\right\rangle$ and $m=2$ and \eqref{prod-1}, it is straightforward to conclude the identity.
\end{proof}

\section{proof of the main theorems}
In this section, we prove our main results. We begin with the proof of Theorem \ref{theorem-DIK-1}.

\begin{proof}[Proof of theorem \ref{theorem-DIK-1}] Let $P(x,y)=y^2-x^3-3\lambda(x+1)^2.$ Then we have 
\begin{align}\label{eqn-main}
p\cdot(|E^{\dik}_{\lambda}(\F_p)|-1)&=p^2+\sum\limits_{x,y,z\in\F_p, z\neq0}\theta(P(x,y)z)\notag\\
&=p^2-1+\sum_{yz\neq0}\theta(zy^2-3\lambda z)+\sum_{xz\neq0}\theta(-zx^3-3\lambda zx^2-6\lambda zx-3\lambda z)\notag\\
&+\sum_{xyz\neq0}\theta(zy^2-zx^3-3\lambda zx^2-6\lambda zx-3\lambda z)\notag\\
&=p^2-1+S_1+S_2+S_3,
\end{align}
where $S_1=\sum\limits_{yz\neq0}\theta(zy^2-3\lambda z), S_2=\sum\limits_{xz\neq0}\theta(-zx^3-3\lambda zx^2-6\lambda zx-3\lambda z)$ and $S_3=\sum\limits_{xyz\neq0}\theta(zy^2-zx^3-3\lambda zx^2-6\lambda zx-3\lambda z).$
Now, using Lemma \ref{lemma-fuselier} and orthogonality relation (Lemma \ref{orthogonality}) we obtain that $$S_1=1+p\phi(3\lambda).$$ 
Again, using  Lemma \ref{lemma-fuselier} we rearrange the terms and make the transformation $z\rightarrow-z$ to obtain  
\begin{align}\label{S-2}
S_2=\frac{1}{(p-1)^4} \sum_{r,s,t,k=0}^{p-2}g(\overline{\chi}^{r}) g(\overline{\chi}^{s}) g(\overline{\chi}^{t}) g(\overline{\chi}^{k})
    \chi^{s+t+k} (3\lambda ) \chi^t (2)\sum_{x\neq0}\chi^{3r+2s+t}(x)\sum_{z\neq0} \chi^{r+s+t+k} (z).
\end{align}
Also, using Lemma \ref{lemma-fuselier}, orthogonality relation and making the transformation $z\rightarrow-z,$ we have 
\begin{align}\label{S-3}
S_3&=-S_2+\frac{1}{(p-1)^4} \sum_{r,s,t,k=0}^{p-2}  g(\overline{\chi}^{r})g(\overline{\chi}^{s}) g(\overline{\chi}^{t}) g(\overline{\chi}^{k}) g(\phi)\chi^t(2)  \chi^{s+t+k} (3\lambda) \phi(-1) \notag\\ 
&\times \sum_{z\neq0} \phi\chi^{r+s+t+k} (z)    \sum_{x\neq0} \chi^{3r+2s+t} (x).
\end{align}
The summation present in $S_3$ is nonzero only if $r+s+t+k+\frac{p-1}{2}\equiv0\pmod{p-1}$ and $3r+2s+t\equiv0\pmod{p-1}.$ Solving these two we have  $k =2r+s+\frac{p-1}{2}$ and substituting this in \eqref{S-3} we have 
\begin{align}\label{S-3-final}
S_3&= -S_2+ \frac{\phi(-3\lambda)g(\phi)}{(p-1)^2} \sum_{r,s=0}^{p-2} g(\overline{\chi}^{r}) g(\overline{\chi}^{s}) g(\chi^{3r+2s}) g\left(\overline{\chi}^{2r + s + \frac{p-1}{2}}\right) \overline{\chi}^{3r+2s}(2)\overline{\chi}^{r}(3\lambda)\notag\\
&=-S_2+S_3',
\end{align}
where $S_3'=\frac{\phi(-3\lambda)g(\phi)}{(p-1)^2} \sum\limits_{r,s=0}^{p-2} g(\overline{\chi}^{r}) g(\overline{\chi}^{s}) g(\chi^{3r+2s}) g\left(\overline{\chi}^{2r + s + \frac{p-1}{2}}\right) \overline{\chi}^{3r+2s}(2)\overline{\chi}^{r}(3\lambda).$
Now, replacing $\chi$ by Teichm\"{u}ller character $\overline{\omega}$ in the expression of $S_3'$ and applying Gross-Koblitz formula we have 
\begin{align}
S_3'&=\frac{(-p)^{-\frac{1}{2}}g(\phi) \phi(-3\lambda)}{(p-1)^2} \sum_{r,s=0}^{p-2}(-p)^{-\left\lfloor\frac{-r}{p-1}\right\rfloor-\left\lfloor\frac{-s}{p-1}\right\rfloor-\left\lfloor\frac{3r+2s}{p-1}\right\rfloor-\left\lfloor\frac{-2r-s}{p-1}-\frac{1}{2}\right\rfloor}\Gamma_p\left(\left\langle\frac{-r}{p-1}\right\rangle\right) \Gamma_p\left(\left\langle\frac{-s}{p-1}\right\rangle\right)\notag \\
&\times \Gamma_p\left(\left\langle\frac{3r+2s}{p-1}\right\rangle\right)\Gamma_p\left(\left\langle\frac{-2r-s}{p-1}-\frac{1}{2}\right\rangle\right){\omega}^{3r+2s}(2){\omega}^{r}(3\lambda). \notag
\end{align}
Transforming $s$ by $s-r$, we have
\begin{align}\label{eq-connect}
S_3'&=\frac{(-p)^{-\frac{1}{2}}g(\phi) \phi(-3\lambda)}{(p-1)^2} 
\sum_{r,s=0}^{p-2}(-p)^{-\left\lfloor\frac{-r}{p-1}\right\rfloor-\left\lfloor\frac{r-s}{p-1}\right\rfloor-\left\lfloor\frac{r+2s}{p-1}\right\rfloor-\left\lfloor\frac{-r-s}{p-1}-\frac{1}{2}\right\rfloor}\Gamma_p\left(\left\langle\frac{-r}{p-1}\right\rangle\right)\notag\\ 
  &\Gamma_p\left(\left\langle\frac{-s+r}{p-1}\right\rangle\right)\Gamma_p\left(\left\langle\frac{r+2s}{p-1}\right\rangle\right)\Gamma_p\left(\left\langle\frac{-r-s}{p-1}-\frac{1}{2}\right\rangle\right){\omega}^{r+2s}(2) {\omega}^{r}(3\lambda).
\end{align}
To simplify further, if we use \eqref{prod-2} for $x=\frac{r+2s}{p-1}$ and $m=2$ then
 $$\Gamma_p\left(\left\langle\frac{r+2s}{p-1}\right\rangle\right)\Gamma_p\left(\frac{1}{2}\right)=\overline{\omega}(2^{r+2s})\Gamma_p\left(\left\langle\frac{r}{2(p-1)}+\frac{s}{p-1}\right\rangle\right)\Gamma_p\left(\left\langle\frac{r}{2(p-1)}+\frac{s}{p-1}+\frac{1}{2}\right\rangle\right).$$
 Substituting this in \eqref{eq-connect} and using Lemma \ref{exponent-1} we may write
  \begin{align}
 S_3'&=\frac{(-p)^{-\frac{1}{2}}g(\phi) \phi(-3\lambda)}{(p-1)^2\Gamma_p\left(\frac{1}{2}\right)} \sum_{r,s=0}^{p-2}(-p)^{-\left\lfloor\frac{-r}{p-1}\right\rfloor-\left\lfloor\frac{r-s}{p-1}\right\rfloor-\left\lfloor\frac{r}{2(p-1)}+\frac{s}{p-1}\right\rfloor-\left\lfloor\frac{r}{2(p-1)}+\frac{s}{p-1}+\frac{1}{2}\right\rfloor-\left\lfloor\frac{-r-s}{p-1}-\frac{1}{2}\right\rfloor}
 {\omega}^{r}(3\lambda)\notag\\
 &\Gamma_p\left(\left\langle\frac{-r}{p-1}\right\rangle\right)\Gamma_p\left(\left\langle\frac{-s+r}{p-1}\right\rangle\right)\Gamma_p\left(\left\langle\frac{r}{2(p-1)}+\frac{s}{p-1}\right\rangle\right)\Gamma_p\left(\left\langle\frac{r}{2(p-1)}+\frac{s}{p-1}+\frac{1}{2}\right\rangle\right)\notag\\&\Gamma_p\left(\left\langle\frac{-r-s}{p-1}-\frac{1}{2}\right\rangle\right)
   \frac{\Gamma_p\left(\left\langle\frac{r}{p-1}\right\rangle\right)\Gamma_p\left(\left\langle\frac{r}{2(p-1)}\right\rangle\right)\Gamma_p\left(\left\langle\frac{r}{2(p-1)}+\frac{1}{2}\right\rangle\right) \Gamma_p\left(\left\langle\frac{-r}{p-1} -\frac{1}{2}\right\rangle\right)}{\Gamma_p\left(\left\langle\frac{r}{p-1}\right\rangle\right)\Gamma_p\left(\left\langle\frac{r}{2(p-1)}\right\rangle\right)\Gamma_p\left(\left\langle\frac{r}{2(p-1)}+\frac{1}{2}\right\rangle\right) \Gamma_p\left(\left\langle\frac{-r}{p-1} -\frac{1}{2}\right\rangle\right)}.\notag
 \end{align}
 Now, using \eqref{relation-1} in the last expression we reduce it to
  \begin{align}
 S_3'&=  \frac{p^2\phi(-3\lambda)}{(p-1)^2} \sum_{r,s=0}^{p-2}(-p)^{\left\lfloor-\frac{1}{2}-\frac{r}{p-1}\right\rfloor-\left\lfloor\frac{-r}{p-1}\right\rfloor-\left\lfloor\frac{r-s}{p-1}\right\rfloor-\left\lfloor\frac{r}{2(p-1)}+\frac{s}{p-1}\right\rfloor-\left\lfloor\frac{r}{2(p-1)}+\frac{s}{p-1}+\frac{1}{2}\right\rfloor-\left\lfloor\frac{-r-s}{p-1}-\frac{1}{2}\right\rfloor}\notag \\
 &\times\Gamma_p\left(\left\langle\frac{r}{p-1}\right\rangle\right)
  \Gamma_p\left(\left\langle\frac{-r}{p-1}\right\rangle\right)
  \frac{\Gamma_p\left(\left\langle\frac{-s+r}{p-1}\right\rangle\right)}{\Gamma_p\left(\left\langle\frac{r}{p-1}\right\rangle\right)}
 \frac{ \Gamma_p\left(\left\langle\frac{r}{2(p-1)}+\frac{s}{p-1}\right\rangle\right)}{\Gamma_p\left(\left\langle\frac{r}{2(p-1)}\right\rangle\right)}\notag\\
 &\times
 \frac{\Gamma_p\left(\left\langle\frac{r}{2(p-1)}+\frac{s}{p-1}+\frac{1}{2}\right\rangle\right)}{\Gamma_p\left(\left\langle\frac{r}{2(p-1)}+\frac{1}{2}\right\rangle\right)}
 \frac{\Gamma_p\left(\left\langle\frac{-r-s}{p-1}-\frac{1}{2}\right\rangle\right)}{\Gamma_p\left(\left\langle\frac{-r}{p-1} -\frac{1}{2}\right\rangle\right)}{\omega}^r(-6\lambda) \binom{\overline{\omega}^r}{\overline{\omega}^r\phi}.
 \end{align}
 If we use Lemma \ref{exponent-3} and \eqref{prod-1} in the above expression for $r\neq0$, then simplifying the terms we have 
  \begin{align}
 S_3'&=\frac{p\phi(-3\lambda)}{(p-1)^2}\sum_{s=0}^{p-2}(-p)^{-1-\left\lfloor\frac{-s}{p-1}\right\rfloor-\left\lfloor\frac{s}{p-1}\right\rfloor-\left\lfloor\frac{s}{p-1}+\frac{1}{2}\right\rfloor-\left\lfloor-\frac{1}{2}-\frac{s}{p-1}\right\rfloor}\Gamma_p\left(\left\langle\frac{-s}{p-1}\right\rangle\right)\\&\Gamma_p\left(\left\langle\frac{s}{p-1}\right\rangle\right)\Gamma_p\left(\left\langle\frac{s}{p-1}+\frac{1}{2}\right\rangle\right)\Gamma_p\left(\left\langle\frac{-s}{p-1}-\frac{1}{2}\right\rangle\right) -
 \frac{p^2\phi(-3\lambda)}{(p-1)^2} \sum_{r=1}^{p-2}{\omega}^r(6\lambda) \binom{\overline{\omega}^r}{\overline{\omega}^r\phi}\notag\\&\times \sum_{s=0}^{p-2}(-p)^{1-\left\lfloor\frac{r-s}{p-1}\right\rfloor-\left\lfloor\frac{r}{2(p-1)}+\frac{s}{p-1}\right\rfloor-\left\lfloor\frac{r}{2(p-1)}+\frac{s}{p-1}+\frac{1}{2}\right\rfloor-\left\lfloor\left\langle-\frac{1}{2}-\frac{r}{p-1}\right\rangle-\frac{s}{p-1}\right\rfloor}\notag\\
 &\times
  \frac{\Gamma_p\left(\left\langle\frac{-s+r}{p-1}\right\rangle\right)}{\Gamma_p\left(\left\langle\frac{r}{p-1}\right\rangle\right)}
 \frac{ \Gamma_p\left(\left\langle\frac{r}{2(p-1)}+\frac{s}{p-1}\right\rangle\right)}{\Gamma_p\left(\left\langle\frac{r}{2(p-1)}\right\rangle\right)}
   \frac{\Gamma_p\left(\left\langle\frac{r}{2(p-1)}+\frac{s}{p-1}+\frac{1}{2}\right\rangle\right)}{\Gamma_p\left(\left\langle\frac{r}{2(p-1)}+\frac{1}{2}\right\rangle\right)}
 \frac{\Gamma_p\left(\left\langle\frac{-r-s}{p-1}-\frac{1}{2}\right\rangle\right)}{\Gamma_p\left(\left\langle\frac{-r}{p-1} -\frac{1}{2}\right\rangle\right)}.\notag
\end{align} 
	Furthermore, using Lemma \ref{relation-2} we have
	\begin{align}\label{eqn-100}
	 S_3'&=\frac{-p(p-2)\phi(3\lambda)}{p-1}-\frac{p^3\phi(-3\lambda)}{p-1}\sum\limits_{r=1}^{p-2}\omega^r(6\lambda)
	 \binom{\overline{\omega}^r}{\overline{\omega}^r\phi}{}_2G_2\left[\begin{array}{cc}
  \frac{r}{p-1}   & -\frac{1}{2} -\frac{r}{p-1}  \\
  \frac{-r}{2(p-1)}    & -\frac{r}{2(p-1)} -\frac{1}{2} \end{array}|\ 1 \right]_p.
	\end{align}
    	Finally substituting \eqref{eqn-100} into \eqref{S-3-final} and then using the expression for $S_1$ in \eqref{eqn-main}
we derive the result.	
\end{proof}

\begin{proof}[Proof of Theorem \ref{theorem-DIK-2}]
 Recalling \eqref{eq-connect} we have 
\begin{align}\label{eqn-1000}
S_3'&=\frac{(-p)^{-\frac{1}{2}}g(\phi) \phi(-3\lambda)}{(p-1)^2} 
\sum_{r,s=0}^{p-2}(-p)^{-\left\lfloor\frac{-r}{p-1}\right\rfloor-\left\lfloor\frac{r-s}{p-1}\right\rfloor-\left\lfloor\frac{r+2s}{p-1}\right\rfloor-\left\lfloor\frac{-r-s}{p-1}-\frac{1}{2}\right\rfloor}\Gamma_p\left(\left\langle\frac{-r}{p-1}\right\rangle\right)\notag\\ 
  &\Gamma_p\left(\left\langle\frac{-s+r}{p-1}\right\rangle\right)\Gamma_p\left(\left\langle\frac{r+2s}{p-1}\right\rangle\right)\Gamma_p\left(\left\langle\frac{-r-s}{p-1}-\frac{1}{2}\right\rangle\right){\omega}^{r+2s}(2) {\omega}^{r}(3\lambda).
\end{align}
Now, substituting \eqref{eqn-1000} into \eqref{S-3-final} and then using the expression for $S_1$ in \eqref{eqn-main} we have 
\begin{align}\label{eqn-1001}
a_p(E^{\dik}_{\lambda})&=-\phi(3\lambda) -
\frac{(-p)^{-\frac{1}{2}}g(\phi) \phi(-3\lambda)}{p(p-1)^2} 
\sum_{r,s=0}^{p-2}(-p)^{-\left\lfloor\frac{-r}{p-1}\right\rfloor-\left\lfloor\frac{r-s}{p-1}\right\rfloor-\left\lfloor\frac{r+2s}{p-1}\right\rfloor-\left\lfloor\frac{-r-s}{p-1}-\frac{1}{2}\right\rfloor}\Gamma_p\left(\left\langle\frac{-r}{p-1}\right\rangle\right)\notag\\ 
  &\Gamma_p\left(\left\langle\frac{-s+r}{p-1}\right\rangle\right)\Gamma_p\left(\left\langle\frac{r+2s}{p-1}\right\rangle\right)\Gamma_p\left(\left\langle\frac{-r-s}{p-1}-\frac{1}{2}\right\rangle\right){\omega}^{r+2s}(2) {\omega}^{r}(3\lambda).
\end{align}
It is easy to see that, for $0\leq s< \frac{p-1}{2}$, we have $\left\lfloor\frac{2s}{p-1}+\frac{r}{p-1}\right\rfloor=\left\lfloor\left\langle\frac{2s}{p-1}\right\rangle+\frac{r}{p-1}\right\rfloor$ and for $\frac{p-1}{2}\leq s\leq p-2$ we have 
$\left\lfloor\frac{2s}{p-1}+\frac{r}{p-1}\right\rfloor=1+\left\lfloor\left\langle\frac{2s}{p-1}\right\rangle+\frac{r}{p-1}\right\rfloor.$
Now, using these along with Lemma \ref{exponent-2}, \eqref{prod-1} and Lemma \ref{lemma-100} in \eqref{eqn-1001}, we have 
\begin{align}\label{eqn-1002}
a_p(E^{\dik}_{\lambda})&=-\frac{p+1}{p-1}\phi(3\lambda)-\frac{p\phi(3\lambda)\phi(1-6\lambda)}{p-1}-\frac{p\phi(2)\phi(6\lambda-1)}{p-1}-\frac{p\phi(3\lambda)}{p-1} \notag\\
&\times \sum_{s=1,s\neq \frac{p-1}{2}}^{p-2}
{_2G_2}\left[\begin{array}{cc}
  0   & -\frac{1}{2} -\frac{s}{p-1} \\
  \frac{s}{(p-1)}    & -\frac{2s}{(p-1)} \end{array}|\ \frac{1}{6\lambda} \right]_p.
\end{align}
Now, multiplying both sides of \eqref{eqn-1002} by $(1-p)\phi(3\lambda)$ we complete the proof.
\end{proof}

\begin{proof}[Proof of Theorem \ref{theorem-summation-1}]
Consider the elliptic curve given by $$E: y^2=x^3+ax+b,$$ where $a=6\lambda-3\lambda^2$ and $b=2\lambda^3-6\lambda^2+3\lambda$ such that $\lambda\in\F_p^{\times}$ and $\lambda\neq2,\frac{3\pm\sqrt{3}}{2}.$ If we make the following transformation $x\rightarrow x+\lambda$ and $y\rightarrow y$, then we have the equivalent curve $$  E^{\dik}_{\lambda} : y^2 = x^3 + 3\lambda(x + 1)^2 .$$ Therefore, we have 
\begin{align}\label{eq-1021}
|E(\F_p)|=| E^{\dik}_{\lambda}(\F_p)|.
\end{align} If we use \cite[Theorem 1.2]{McCarthyPJM}, then we have 
\begin{align}\label{eq-1020}
a_p(E)=\phi(2\lambda^3-6\lambda^2+3\lambda)\cdot p\cdot {_2G_2}\left[\begin{array}{cc}
  \frac{1}{4}   & \frac{3}{4}  \\
  \frac{1}{3}    & \frac{2}{3}  \end{array}|\frac{(2\lambda^2-6\lambda+3)^2}{4\lambda(\lambda-2)^3}\right]_p,
\end{align}
where $a_p(E)=p+1-|E(\F_p)|.$ Finally, using Theorem \ref{theorem-DIK-1}, \eqref{eq-1020} and \eqref{eq-1021} we deduce the result.
\end{proof}

\begin{proof}[Proof of Theorem \ref{theorem-summation-2}]
The proof of the theorem is same as the proof of Theorem \ref{theorem-summation-1}. To obtain the identity we use Theorem \ref{theorem-DIK-2}, \eqref{eq-1020} and \eqref{eq-1021}.

\end{proof}

\begin{proof}[Proof of Theorem \ref{theorem-Jacobi}] We first prove the first equality. To do this, we 
employ similar method as use in proof of Theorem \ref{theorem-DIK-1} and derive that
\begin{align}
 -p(a_p(E_\lambda^{\jac}))&=p+\frac{g(\phi)\phi(-1)}{(p-1)^3}\sum_{r,s,t=0}^{p-2}g(T^{-r})g(T^{-s})g(T^{-t})T^s(2\lambda)\sum_{z}\phi T^{r+s+t}(z)\sum_{x}T^{4r+2s}(x).\notag
\end{align}
Using orthogonality of multiplicative characters it is easy to see that the last sum present in the above identity is nonzero if  $s=-2r$ and $s=-2r+\frac{p-1}{2}.$ Substituting these values in the above identity we have 
\begin{align}\label{eqn-1004}
-p\cdot a_p(E^{\jac}_{\lambda})&=p+A_1+A_2,
\end{align}
\noindent where $$A_1=\frac{g(\phi)\phi(-1)}{(p-1)^2}\sum\limits_{r,t=0}^{p-2}g(T^{-r})g(T^{2r})g(T^{-t})T^{-2r}(2\lambda)\sum_{z}T^{t-r}(z)\phi(z)$$ and 
where $$A_2=\frac{g(\phi)\phi(-2\lambda)}{(p-1)^2}\sum_{r,t=0}^{p-2}g(T^{-r})g(T^{2r}\phi)g(T^{-t})T^{-2r}(2\lambda)\sum_{z\in \mathbb F_p^*}T^{t-r}(z).$$
Now, using orthogonality relation of multiplicative characters and then applying Hasse-Davenport relation in the expression of $A_1,$ we deduce that $$A_1=-2p.$$
\noindent Similarly, using orthogonality relation of multiplicative characters we obtain
\begin{align}\label{eqn-1005}
  A_2&=\frac{1}{(p-1)}\sum_{r=0}^{p-2}g(T^{-r})g(T^{2r}\phi)g(T^{-r})g(\phi)\phi(-2\lambda)T^{r}\left(\frac{1}{4\lambda^2}\right).
\end{align}
Again, by making use of Hasse-Davenport formula we have $$g(T^{2r}\phi)=\frac{g(T^{4r}){T}^{-4r}(2)g(\phi)}{g(T^{2r})}.$$
Substituting this into \eqref{eqn-1005} and using \eqref{inverse} we have 
\begin{align}\label{eqn-1006}
  A_2&=\frac{p\phi(2\lambda)}{(p-1)}\sum_{r=0}^{p-2}g(T^{-r})^2\frac{g(T^{4r})}{g(T^{2r})}T^{r}\left(\frac{1}{2^6\lambda^2}\right).
\end{align}
Now, replacing $T$ by $\overline{\omega}$ and applying Gross-Koblitz formula we have 
\begin{align}\label{eqn-1010}
A_2=\frac{p \phi(2\lambda)}{(p-1)}\sum_{r=0}^{p-2}(-p)^{-\left\lfloor\frac{-r}{p-1}\right\rfloor-\left\lfloor\frac{-r}{p-1}\right\rfloor-\left\lfloor\frac{4r}{p-1}\right\rfloor+\left\lfloor\frac{2r}{p-1}\right\rfloor}
\frac{ \Gamma_p\left(\left\langle\frac{-r}{p-1}\right\rangle\right) \Gamma_p\left(\left\langle\frac{-r}{p-1}\right\rangle\right)\Gamma_p\left(\left\langle\frac{4r}{p-1}\right\rangle\right)\overline{\omega}^{r}(\frac{1}{2^6\lambda^2})}{\Gamma_p\left(\left\langle\frac{2r}{p-1}\right\rangle\right)}.
\end{align}
Now, using the product formula given in \eqref{prod-3} we have 
\begin{align}\label{eqn-1007}
   \Gamma_p\left(\left\langle\frac{4r}{p-1}\right\rangle\right)=\frac{\Gamma_p\left(\left\langle\frac{r}{p-1}\right\rangle\right)\Gamma_p\left(\left\langle\frac{r}{p-1}+\frac{1}{4}\right\rangle\right)\Gamma_p\left(\left\langle\frac{r}{p-1}+\frac{1}{2}\right\rangle\right)\Gamma_p\left(\left\langle\frac{r}{p-1}+\frac{3}{4}\right\rangle\right) {\overline\omega^r(4^4)}}{\Gamma_p\left(\left\langle\frac{1}{4}\right\rangle\right)\Gamma_p\left(\left\langle\frac{1}{2}\right\rangle\right)\Gamma_p\left(\left\langle\frac{3}{4}\right\rangle\right)} 
\end{align}
and
\begin{align}\label{eqn-1008}
   \Gamma_p\left(\left\langle\frac{2r}{p-1}\right\rangle\right)=\frac{\Gamma_p\left(\left\langle\frac{r}{p-1}\right\rangle\right)\Gamma_p\left(\left\langle\frac{r}{p-1}+\frac{1}{2}\right\rangle\right) {\overline\omega^r(4)}}{\Gamma_p\left(\left\langle\frac{1}{2}\right\rangle\right)}.
\end{align}
Furthermore, it is easy to see that for $0\leq r\leq p-2$ we have
\begin{equation}\label{eqn-1009}
    \left\lfloor\frac{4r}{p-1}\right\rfloor-\left\lfloor\frac{2r}{p-1}\right\rfloor=\left\lfloor\frac{1}{4}+\frac{r}{p-1}\right\rfloor+\left\lfloor\frac{3}{4}+\frac{r}{p-1}\right\rfloor.
\end{equation}
Substituting \eqref{eqn-1007} , \eqref{eqn-1008} and \eqref{eqn-1009} into \eqref{eqn-1010} and taking the transformation $r\rightarrow-r$ we deduce that
\begin{align}\label{eqn-1011}
A_2=-p\cdot  \phi(2\lambda)\cdot {_2G_2}\left[\begin{array}{cc}
 \frac{1}{4}    & \frac{3}{4} \\
  0    & 0  \end{array}|{\lambda^2}\right]_p.
\end{align}
Finally, substituting the values of $A_1$ and $A_2$ into \eqref{eqn-1004} we have the first equality. To obtain the second equality, we make the transformation $r\rightarrow -r-\frac{p-1}{2}$ in \eqref{eqn-1006} and proceeding similar steps as in the proof of the first equality we complete the proof.
\end{proof}

\begin{proof}[Proof of Theorem \ref{theorem-Jacobi trans}]
Consider the Jacobi curve $$E^{\jac}_{\lambda}: v^2=u^4+2\lambda u^2+1,$$ where $\lambda\in\F_p^{\times}$ such that $\lambda\neq\pm1.$  For each point $(u,v)\in E^{\jac}_{\lambda}(\F_p)$ such that $u\neq0$ we have that $(x,y)= \left(\frac{2(v+1)}{u^2},\frac{4(1+v+\lambda u^2)}{u^3}\right)$ is a point on the curve $E'$ given by $$E': y^2=x^3+2\lambda x^2-4x-8\lambda.$$ Moreover, for $u=0,$ there are two points on the curve $E^{\jac}_{\lambda}$ that do not correspond to any points on the curve $E'.$
Conversely, for any point $(x,y)\in E'(\F_p)$ such that $x\neq-2\lambda$ and $y\neq0,$ we have that $(u,v)=\left(\frac{2(x+2\lambda)}{y}, \frac{2(x+2\lambda)^2x}{y^2}-1\right)$ is a point on the curve $E^{\jac}_{\lambda}.$
Furthermore, if $x=-2\lambda,$ then $y=0$ and there are three points on the curve $E',$ namely $(-2\lambda,0),(\pm2,0)$ that do not correspond to any point of the curve $E^{\jac}_{\lambda}.$
Furthermore, if we use the change of variables $(x,y)\rightarrow({x-2\lambda},y)$ in the curve $E'$, then it is equivalent to the elliptic curve $E''$ given by 
$$E'':y^2=x^3-4\lambda x^2+(4\lambda ^2-4)x.$$
Using all the above arguments we conclude that $|E^{\jac}_{\lambda}(\F_p)|-2=|E''(\F_p)|-3.$ This yields,
\begin{align}\label{eqn-1024}
 a_p(E^{\jac}_\lambda)=a_p(E'')+1.
 \end{align}
 Now, if we use \cite[Theorem 3.5]{BS}, for $\lambda\neq0\pm1,\pm\sqrt{2}$ then we have 
\begin{align}\label{eqn-1023}
a_p(E'')=p\phi(1-\lambda^2){_2G_2}\left[\begin{array}{cc}
  \frac{1}{2}   & \frac{1}{2}  \\
  \frac{1}{4}    & \frac{3}{4}  \end{array}|\frac{\lambda^2-1}{\lambda^2}\right]_p.
\end{align}
Moreover, it is easy to see that 
\begin{align}\label{eqn-inverse}
{_2G_2}\left[\begin{array}{cc}
  \frac{1}{2}   & \frac{1}{2}  \\
  \frac{1}{4}    & \frac{3}{4}  \end{array}|\frac{\lambda^2-1}{\lambda^2}\right]_p={_2G_2}\left[\begin{array}{cc}
\frac{1}{4}   & \frac{3}{4} \\
  \frac{1}{2}    & \frac{1}{2}  \end{array}|\frac{\lambda^2}{\lambda^2-1}\right]_p.
  \end{align}
Finally, combining Theorem \ref{theorem-Jacobi}, \eqref{eqn-1024}, \eqref{eqn-1023} and \eqref{eqn-inverse} we deduce the result.
\end{proof}
\begin{proof}[Proof of Theorem \ref{Hessian-trans}]
Here, we consider the Hessian form of elliptic curve given by $$E^{\hes}_d:u^3+v^3+1-3duv=0,$$ where 
$d\neq-2,0$ and $d^3\neq1$ such that $\frac{d^2+d+1}{3(d+2)}$ is a square in $\F_p^{\times}.$ Now,  applying \cite[Theorem 3.3]{BS1}, we have 
\begin{align}\label{eq-1024}
\#E^{\hes}_d(\F_p)=\gamma-1+p-p\phi(-3d){}_2G_2\left[\begin{array}{cc}
  \frac{1}{2}   & \frac{1}{2}  \\
  \frac{1}{6}    & \frac{5}{6}  \end{array}|\frac{1}{d^3}\right]_p,
\end{align}
where $\gamma=\begin{cases}
    5-6\phi(-3), & \text{if} \ p \equiv 1 \pmod{3},\\1, &  \text{if}\ p\not\equiv 1 \pmod{3},
\end{cases}$

\noindent and $\#E^{\hes}_d(\F_p)$ denotes the number of $\F_p$-rational points of the curve $E^{\hes}_d.$ Now, if $u+v\neq-d,$ and $(u,v)\in E^{\hes}_d(\F_p),$ then $(x,y)=\left(-\frac{k(u+v+1)}{u+v+d},-\frac{k(d-1)(u-v)}{u+v+d}\right),$ is a point on the curve 
$$\widetilde{E}:y^2=x^3+\{(d+2)x+k\}^2,$$ where $k=\frac{4}{3}(d^2+d+1).$  Moreover, we can easily verify that if $u+v=-d,$ then $(u,v)$ is not a point on the Hessian curve $E^{\hes}_d.$ 
Conversely, if $(x,y)\in \widetilde{E}(\F_p)$ such that $x\neq-k,$ then $(u,v)=\left(\frac{\frac{d}{8}\alpha x+\frac{1}{8}\alpha y+\frac{1}{2}\beta}{-\frac{1}{4}\alpha x-\beta},\frac{\frac{d}{8}\alpha x-\frac{1}{8}\alpha y+\frac{1}{2}\beta}{-\frac{1}{4}\alpha x-\beta}\right)$ is a point on the curve $E^{\hes}_d$, where $\alpha=\frac{3}{d^3-1}\  \text{and} \  \beta=\frac{1}{d-1}.$

\noindent Furthermore, for $x=-k,$ there are $N_0$ number of points on the curve $\widetilde{E}$ that do not correspond to any point on the curve $E^{\hes}_d,$ where 
$$N_0=\begin{cases}
    2, & \text{if} \ p \equiv 1,-5 \pmod{12},\\0, &  \text{if}\ p\not\equiv 1,-5 \pmod{12}.
\end{cases}$$
\noindent Therefore, we have the following relation
\begin{align}\label{eqn-1025}
\#E^{\hes}_d(\F_p)=\#\widetilde{E}(\F_p)-N_0,
\end{align}
where $\#\widetilde{E}(\F_p)$ denotes the number of $\F_p$-rational points of the curve $\widetilde{E}.$
Furthermore, by making the transformation $(x,y)\rightarrow(\frac{k}{d+2}x, \frac{k}{d+2}\sqrt{\frac{k}{d+2}}\ y)$ in the curve $\widetilde{E}$ we obtain that $\widetilde{E}$ is equivalent to the curve $E^{\dik}_{\lambda},$ where $\lambda=\frac{(d+2)^3}{3k}.$ This yields,
\begin{align}\label{eqn-5000}
\# E^{\dik}_{\lambda}(\F_p)=\#\widetilde{E}(\F_p).
\end{align}
Hence, using \eqref{eqn-1025} into \eqref{eqn-5000} we have,
$\#E^{\hes}_d(\F_p)=\#{E^{\dik}_\lambda}(\F_p)-N_0.\notag
$
This yields, 
\begin{align}\label{eqn-1026}
 \#E^{\hes}_d(\F_p)=p-a_p(E^{\dik}_\lambda)-N_0.
 \end{align}
Finally, combining Theorem \ref{theorem-DIK-1}, \eqref{eq-1024} and \eqref{eqn-1026} we deduce the result.
\end{proof}

\begin{proof}[Proof of Theorem \ref{Hessian-trans2}]
Proceeding similar steps as in the proof of Theorem \ref{Hessian-trans} and 
combining Theorem \ref{theorem-DIK-2}, \eqref{eq-1024} and \eqref{eqn-1026} we derive the result.
\end{proof}
\section{Acknowledgement}
The second author is supported by Science and Engineering Research Board/Anusandhan National Research Foundation [CRG/2023/003037].

\end{document}